\documentclass[a4paper, 10pt]{amsart}
\usepackage[top=80pt,bottom=65pt,left=70pt,right=70pt]{geometry}
\usepackage{graphicx, eepic}
\usepackage{times}
\normalfont
\usepackage{bbm}
\usepackage{xcolor}

\usepackage{amsthm,amsmath,amsfonts,amssymb}
\usepackage{hyperref}
\hypersetup{colorlinks}
\hypersetup{
    bookmarks=true,         
    unicode=false,          
    pdftoolbar=true,        
    pdfmenubar=true,        
    pdffitwindow=false,     
    pdfstartview={FitH},    
    pdftitle={My title},    
    pdfauthor={Author},     
    pdfsubject={Subject},   
    pdfcreator={Creator},   
    pdfproducer={Producer}, 
    pdfkeywords={keyword1} {key2} {key3}, 
    pdfnewwindow=true,      
    colorlinks=true,       
    linkcolor=blue,          
    citecolor=red,        
    filecolor=violet,      
    urlcolor=violet       
}
\usepackage{multirow, url}
\usepackage{color}
\usepackage[all]{xy}
\newtheorem{thm}{Theorem}[section]

\newtheorem{lem}[thm]{Lemma}

\theoremstyle{definition}
\newtheorem{defn}[thm]{Definition}
\theoremstyle{remark}
\newtheorem{rem}[thm]{Remark}

\numberwithin{equation}{section} \numberwithin{table}{section}

\newcommand{\zell}{{\Z/{\ell\Z}}}

\newcommand{\arsub}{\ar@{}[r]|-*[@]{\subset}}
\newcommand{\arsup}{\ar@{}[r]|-*[@]{\supset}}
\newcommand{\arcap}{\ar@{}[d]|-*[@]{\subset}}
\newcommand{\arcup}{\ar@{}[u]|-*[@]{\subset}}

\renewcommand{\pmod}[1]{{~(\mathrm{mod}~{#1})}}
\renewcommand{\mod}[1]{{~\mathrm{mod}~{#1}}}

\newcommand{\Q}{{\mathbb{Q}}}
\newcommand{\Z}{{\mathbb{Z}}}

\newcommand{\T}{{\mathbb{T}}}

\newcommand{\bP}{{\mathbb{P}}}

\newcommand{\m}{{\mathfrak{m}}}

\newcommand{\cC}{{\mathcal{C}}}

\newcommand{\cI}{{\mathcal{I}}}

\newcommand{\cT}{{\mathcal{T}}}

\newcommand{\Pic}{{\mathrm{Pic}}}
\newcommand{\End}{{\mathrm{End}}}

\newcommand{\old}{{\mathrm{old}}}
\newcommand{\new}{{\mathrm{new}}}

\newcommand{\num}{{\mathrm{num}}}

\newcommand{\modl}{{\pmod {\ell}}} 
\mathchardef\hyp="2D

\renewcommand{\th}{{\mathrm{th}}}

\newcommand{\br}[1]{\langle #1 \rangle}

\newcommand{\zmod}[1]{{\Z/{#1}\Z}}

\newcommand{\exclude}[1]{}

\begin{document}                                                                          

\title{Rational torsion points on Jacobians of modular curves}
\author{Hwajong Yoo}
\address{Center for Geometry and Physics, Institute for Basic Science (IBS), Pohang, Republic of Korea 37673}
\email{hwajong@gmail.com}
\thanks{This work was supported by IBS-R003-G1.}

\subjclass[2010]{11G10, 11G18, 14G05}
\keywords{Rational points, modular curves, Eisenstein ideals}

\begin{abstract}
Let $p$ be a prime greater than 3. Consider the modular curve $X_0(3p)$ over $\Q$ and its Jacobian variety $J_0(3p)$ over $\Q$. Let $\cT(3p)$ and $\cC(3p)$ be the group of rational torsion points on $J_0(3p)$ and the cuspidal group of $J_0(3p)$, respectively. We prove that the $3$-primary subgroups of $\cT(3p)$ and $\cC(3p)$ coincide unless $p\equiv 1 \pmod 9$ and $3^{\frac{p-1}{3}} \equiv 1 \pmod {p}$.
\end{abstract} 
\maketitle
\setcounter{tocdepth}{1}
\tableofcontents

\section{Introduction}
Let $N$ be a square-free integer. Consider the modular curve $X_0(N)$ and its Jacobian variety $J_0(N)=\Pic^0(X_0(N))$. Let $\cT(N)$ denote the group of rational torsion points on $J_0(N)$ and let $\cC(N)$ denote the cuspidal group of $J_0(N)$. By Manin and Drinfeld \cite{Dr73, Ma72}, we have $\cC(N)\subseteq \cT(N)$ and they are both finite abelian groups.

When $N$ is prime, Ogg conjectured that $\cT(N)=\cC(N)$ \cite[Conjecture 2]{Og75}.
In his article \cite{M77}, Mazur proved this conjecture by studying the Eisenstein ideal of level $N$. Recently, Ohta proved a generalization of the result of Mazur \cite{Oh14}. More precisely, he proved the following. 
\begin{thm}[Ohta]\label{thm:ohta}
For a prime $\ell \geq 5$, we have $\cT(N)[\ell^\infty] = \cC(N)[\ell^\infty]$.
Moreover, if 3 does not divide $N$, then  $\cT(N)[3^\infty] = \cC(N)[3^\infty]$.
\end{thm}
(For a finite abelian group $A$, $A[\ell^{\infty}]$ denotes the $\ell$-primary subgroup of $A$.)

We briefly sketch the proof of this theorem.
Let $T_r$ (resp. $U_p$ and $w_p$) denote the $r^\th$ Hecke operator (resp. the $p^\th$ Hecke operator and the Atkin-Lehner operator with respect to $p$) acting on $J_0(N)$ for a prime $r$ not dividing $N$ (resp. a prime divisor $p$ of $N$). Let $\T(N)$ (resp. $\T(N)'$) be the $\Z$-subalgebra of $\End(J_0(N))$ generated by $T_r$'s and $U_p$'s (resp. $T_r$'s and $w_p$'s) for primes $ r\nmid N$ and $p \mid N$. Let
$$
\cI_0 : = ( T_r-r-1 : \text{for primes } r \nmid N)
$$
be the (minimal) Eisenstein ideal of $\T(N)$ (or $\T(N)'$). Then, $\cI_0$ annihilates $\cT(N)$ and $\cC(N)$ by the Eichler-Shimura relation. Thus, $\cT(N)[\ell^\infty]$ is a module over 
$\T(N)_{\ell}/{\cI_0}$ (or $\T(N)'_{\ell}/{\cI_0}$), where $\T(N)_{\ell}:=\T(N)\otimes_{\Z} \Z_{\ell}$. 
Note that since $w_p^2=1$, for a prime $\ell\geq 3$ we have the following decomposition:
$$
\T(N)'_{\ell}/{\cI_0} = \prod_{M\mid N,~M \neq N} \T(N)'_{\ell}/{\cI_{M}},
$$ 
where $\cI_{M}:=(w_p-1, ~w_q+1,~\cI_0 : \text{ for primes } p \mid M \text{ and } q \mid N/M)$. Thus, we have 
$$
\cT(N)[\ell^\infty] = \oplus \cT(N)[\ell^{\infty}] [\cI_M] \quad\text{and}\quad \cC(N)[\ell^\infty] = \oplus \cC(N)[\ell^{\infty}] [\cI_M].
$$
Finally, he proved that $\cT(N)[\ell^\infty][\cI_M] = \cC(N)[\ell^{\infty}][\cI_M]$ by computing the index of $\cI_M$ (up to 2-primary parts).

In this paper, we discuss the case where $N=pq$ for two distinct primes $p$ and $q$.
In contrast to the discussion above, we use $\T(pq)$ instead of $\T(pq)'$ and hence the corresponding decomposition of $\T(pq)/{\cI_0}$ as above does not always exist. 
(However, other computations are relatively easier than the method by Ohta.)
When $\ell$ satisfies some conditions, we get the similar decomposition of the quotient ring $\T(pq)/{\cI_0}$ and we can prove the following. 
\begin{thm}[Main Theorem]\label{thm:maintheorem}
For a prime $\ell$ not dividing $2pq \gcd(p-1,~q-1)$, we have $\cT(pq)[\ell^\infty]=\cC(pq)[\ell^\infty]$.
Moreover, $\cT(pq)[p^\infty]=\cC(pq)[p^\infty]$
if one of the following holds:
\begin{enumerate}
\item $p\geq 5$ and $\begin{cases} 
\text{either } ~q \not\equiv 1 \pmod p   ~~\text{ or} \\
q \equiv 1 \pmod p ~\text{ and }~ p^{\frac{q-1}{p}} \not\equiv 1 \pmod q .
\end{cases}$
\item $p=3$ and $\begin{cases} 
\text{either } ~q \not\equiv 1 \pmod 9   ~~\text{ or} \\
q \equiv 1 \pmod 9 ~\text{ and }~ 3^{\frac{q-1}{3}} \not\equiv 1 \pmod q .
\end{cases}$
\end{enumerate}
\end{thm}

Note that most cases are special ones of Theorem \ref{thm:ohta}. The new result is as follows:
\begin{thm}
Let $p$ be a prime greater than 3. 
Assume that either $p \not\equiv 1 \pmod 9$ or $3^{\frac{p-1}{3}} \not\equiv 1 \pmod p$. Then, we get
$$
\cT(3p)[3^\infty]=\cC(3p)[3^\infty].
$$
\end{thm}

\subsection{Notation}\label{sec:notation}
For $x=a/b \in \Q$, we denote by $\num(x)$ the numerator of $x$, i.e.,
$$
\num(x) := \frac{a}{(a, ~b)}.
$$

From now on, we denote by $\ell^\alpha:=\ell^{\alpha(p,~q,~\ell)}$ (resp. $\ell^\beta:=\ell^{\beta(p,~q,~\ell)}$) the exact power of $\ell$ dividing 
$$
M_p:=\num\left( \frac{(p-1)(q^2-1)}{3} \right) ~\left(\text{ resp. } \num\left( M_q:=\frac{(p^2-1)(q-1)}{3} \right) \right).
$$

\section{Eisenstein ideals of level $pq$}\label{sec:Eisenstein}
Throughout this section, we fix two distinct primes $p$ and $q$; and $\ell$ denotes a prime not dividing $2pq(q-1)$. 
Let $\T:=\T(pq)$ and $\T_{\ell}:=\T(pq)\otimes_{\Z} \Z_{\ell}$. We say that an ideal of $\T$ is
\textit{Eisenstein} if it contains
$$
\cI_0:=(T_r-r-1~:~\text{for primes } r \nmid pq).
$$ 

\begin{defn}
We define Eisenstein ideals as follows: 
$$
\cI_1 := (U_p-1,~U_q-1, ~\cI_0);
$$
$$
\cI_2 := (U_p-1,~U_q-q, ~\cI_0) \quad\text{and}\quad \cI_3 := (U_p-p,~U_q-1,~\cI_0).
$$
Moreover, we set $\m_i :=(\ell, ~\cI_i)$. They are all possible Eisenstein maximal ideals in $\T_{\ell}$ by the result in \cite[\textsection 2]{Yoo15a}. For ease of notation, we set $\T_i :=\T_{\m_i}=\lim_{\leftarrow n} \T/{\m_i^n}$.
\end{defn}

Since $\T_{\ell}$ is a semi-local ring, we have
$$
\T_{\ell} = \prod_{\ell \in \m~ \text{maximal}} \T_{\m}.
$$

Using the above description of Eisenstein maximal ideals, we prove the following:

\begin{thm}\label{thm:decomposition}
The quotient $\T_{\ell}/{\cI_0}$ is isomorphic to $\T_{\ell}/{\cI_2} \times \T_{\ell}/{\cI_3}$.
\end{thm}

This theorem is crucial to deduce our main theorem. In general, the author expects that
$\T_{\ell}/\cI_0$ should be isomorphic to
$$
\{ (x,~y,~z) \in \T_{\ell}/{\cI_1} \times \T_{\ell}/{\cI_2} \times \T_{\ell}/{\cI_3} ~:~ x \equiv y \pmod {p-1} \text { and } x \equiv z \pmod {q-1} \}.
$$
 
Before proving the theorem above, we need several lemmas.
\begin{lem}
We have $(U_p-1)(U_p+1) \in \cI_0 \T_{\ell}$.
\end{lem}
\begin{proof}
Since $q\not\equiv 1 \modl$, any maximal ideal containing $\cI_0$ cannot be $p$-old. Therefore $\T_{\ell}/\cI_0 \simeq \T_{\ell}^{p\hyp\new}/\cI_0$. Since $U_p^2=1$ in $\T_{\ell}^{p\hyp\new}$, the result follows. 
\end{proof}

\begin{lem}\label{lem:decomposition}
Suppose that $\m_2$ is maximal. Then, we have
$$
\T_2/\cI_0 = \T_2/\cI_2 \simeq \T_{\ell}/\cI_2.
$$
If $\m_1$ is maximal, then $p\equiv 1\modl$ and hence $\m_1=\m_3$; moreover, we have $\T_1/\cI_0 = \T_3/\cI_0 \simeq \T_{\ell}/\cI_3$. If $p \not\equiv 1 \modl$, then $\m_1$ is not maximal and $\T_3/\cI_0\simeq \T_{\ell}/\cI_3$.
\end{lem}

\begin{proof}
Suppose that $\m_2$ is maximal.
Since $U_p-1 \in \m_2$ and $\ell$ is odd, $U_p+1 \not\in \m_2$ and hence it is a unit in $\T_2$. By the lemma above, $(U_p-1)(U_p+1) \in \cI_0 \T_{\ell}$ and hence $U_p-1 \in \cI_0 \T_2$. Similarly, we have $U_q-q \in \cI_0 \T_2$ because $q \not\equiv 1 \modl$ and $(U_q-1)(U_q-q) \in \cI_0\T_2$ by the following lemma. Thus, we have $\T_2/\cI_0 = \T_2/\cI_2$. Since the index of $\cI_2$ in $\T$ is finite (cf. \cite[Lemma 3.1]{Yoo14}), we have $\m_2^{n} \subseteq \cI_2$ for large enough $n$. Therefore $\T_{\ell}/{(\m_2^n, \,\cI_2)} \simeq \T_{\ell}/\cI_2$ and hence $\T_2/\cI_2 \simeq \T_{\ell}/\cI_2$.

If $\m_1$ is maximal, the index of $\cI_1$ in $\T$ is divisible by $\ell$. By \cite[Theorem 1.4]{Yoo15a}, it is $\num(\frac{(p-1)(q-1)}{3})$ up to powers of $2$ and hence $p \equiv 1 \modl$. 

Assume that $p\equiv 1 \modl$. Let $\alpha$ be the number in \textsection \ref{sec:notation}. 
Since $\ell$ does not divide $(p+1)(q-1)$, $\ell^\alpha$ divides $(p-1)$.
Note that the index of $\cI_3$ in $\T_{\ell}$ is equal to $\ell^\alpha$ (cf. \cite[Theorem 1.4]{Yoo15a}) and hence $\cI_3 \T_{\ell}$ contains $p-1$. Thus, $U_p-1=(U_p-p)+(p-1) \in \cI_3\T_{\ell}$. In other words, $\cI_1\T_{\ell} \subseteq \cI_3\T_{\ell}$. Similarly, we have $\cI_3\T_{\ell} \subseteq \cI_1\T_{\ell}$. Therefore we have $\cI_1\T_{\ell}=\cI_3\T_{\ell}$. By the same argument as above, $\cI_0\T_3$ contains $U_p-1$ and $(U_q-1)(U_q-q)$. Since $q\not\equiv 1\modl$ and $U_q-1 \in \m_3$, we have $U_q-q \not\in \m_3$ and hence $\T_3/{\cI_0} = \T_3 / \cI_3$. 
By the same argument as above, we get $\T_3/\cI_3 \simeq \T_{\ell}/\cI_3$. 

If $p\not\equiv 1\modl$, then $\m_3$ is neither $p$-old nor $q$-old. If $p\not\equiv -1\modl$, then $\m_3$ is not maximal. Thus, we have $\T_{\ell}/\cI_3=\T_3/\cI_0=0$. If $p\equiv -1\modl$, then the result follows by \cite[Proposition 2.3]{Yoo15}.
\end{proof}

\begin{lem}\label{lem:(U_q-1)(U_q-q)}
Let $I:=(U_p-1, ~\cI_0) \subseteq \T_{\ell}$. Then, we get $(U_q-1)(U_q-q) \in I$.
\end{lem}
\begin{proof}
We closely follow the argument in \cite[\textsection II. 5]{M77}. 

Let $f(z):=\sum_{n\geq 1} (T_n \mod {I}) x^n$ be the Fourier expansion (at $\infty$) of a cusp form of weight 2 and level $pq$ over $\T_{\ell}/I$, where $x=e^{2\pi i z}$.
Let $E:=E_{p, \,pq}$ be an Eisenstein series of weight 2 and level $pq$ in \cite[\textsection 2.3]{Yoo14}. Note that 
$$
(f-E)(z) \equiv (U_q-q)\sum_{n\geq 1} a_n x^{qn} \pmod I,
$$
where $a_p=1$ and $a_r=1+r$ for all primes $r \neq pq$; and $a_q=U_q+q$.   
If $U_q-q \not\in I$, then by Ohta \cite[Lemma 2.1.1]{Oh14}, there is a cusp form $g(z)=\sum_{n\geq 1} b_n x^n$ of weight 2 and level $p$
such that 
$$
(f-E)(z) \equiv (U_q-q)\sum_{n\geq 1} a_n x^{qn} \equiv (U_q-q) g(qz) \pmod I.
$$
Therefore $p\equiv 1 \modl$ and $b_r \equiv 1+r \pmod {I'}$ for primes $r\neq p$, where $I'$ is the Eisenstein ideal of level $p$. Thus, we have $(U_q-q)(a_q-b_q) \equiv (U_q-q)(U_q-1) \in I$.
\end{proof}

Now, we are ready to prove Theorem \ref{thm:decomposition}.
\begin{proof}[Proof of Theorem \ref{thm:decomposition}]
If $p\equiv 1 \modl$, then $\m_1=\m_3$. Otherwise $\m_1$ is not maximal. Therefore, we have
$$
\T_{\ell}/\cI_0 \simeq \T_2/\cI_0 \times \T_3/\cI_0 = \T_2/\cI_2 \times \T_3/\cI_3 
\simeq \T_{\ell}/\cI_2 \times \T_{\ell}/\cI_3.
$$
\end{proof}

\section{Case where $\ell$ does not divide $pq$}\label{sec:proof}
From now on, let $\cC:=\cC(pq)$ and $\cT:=\cT(pq)$ be the cuspidal group of $J_0(pq)$ and the group of rational torsion points on $J_0(pq)$, respectively.
For a prime $r$ and a finite abelian group $A$, we denote by $A[r^\infty]$ the $r$-primary subgroup of $A$.
In this section, we prove the following theorem.
\begin{thm}\label{thm:maintheorem}
For a prime $\ell$ not dividing $2pq(q-1)$, we have $\cT[\ell^\infty] = \cC[\ell^\infty]$.
\end{thm}

Before proving this theorem, we introduce some cuspidal divisors.

Let $P_n$ be the cusp of $X_0(pq)$ corresponding to $1/n \in \bP^1(\Q)$. Let $C_p:=P_1-P_p$ and $C_q:=P_1-P_q$ denote the cuspidal divisors in $\cC$.
Let $M_p = \ell^{\alpha} \times x$ and $M_q=\ell^{\beta} \times y$ as in \textsection \ref{sec:notation}. (Thus, we have $(\ell,~xy)=1$.)
We define 
$$
D_p := xC_p \quad \text{and} \quad D_q:=yC_q.
$$
Then, $\br {D_p}$ (resp. $\br {D_q}$) is a free module of rank 1 over $\T_{\ell}/\cI_2 \simeq \zmod {\ell^{\alpha}}$ (resp. $\T_{\ell}/\cI_3 \simeq \zmod {\ell^{\beta}}$) (cf. \cite[Theorem 1.4]{Yoo15a}). 

Now we prove the Theorem above.

\begin{proof}[Proof of Theorem \ref{thm:maintheorem}]
By the Eichler-Shimura relation, $\cT[\ell^\infty]$ is a module over $\T_{\ell}/\cI_0$. Therefore $\cT[\ell^\infty]$ decomposes into $\cT[\ell^\infty][\cI_2] \times \cT[\ell^\infty][\cI_3]$ by Theorem \ref{thm:decomposition}. Hence it suffices to show that 
$\cT[\ell^\infty][\cI_2] = \br {D_p}$ and $\cT[\ell^\infty][\cI_3] = \br {D_q}$.

If $\alpha=0$, then $\T_{\ell}/\cI_2=0$ and hence $\cT[\ell^\infty][\cI_2]=\br {D_p}=0$. Thus, we may assume that 
$\alpha \geq 1$. 
Note that 
$$
\cT[\ell^\infty][\cI_2] \simeq \prod_{i=1}^t \zmod {\ell^{a_i}},
$$ 
where $1\leq a_i \leq \alpha$ because $\T_{\ell}/\cI_2 \simeq \zmod {\ell^{\alpha}}$ (and $\cT$ is finite).
Since $D_p \in \cT[\ell^\infty]$, we have $\br {D_p} \subseteq \cT[\ell^\infty][\cI_2]$ and hence $t\geq 1$; and $\cT[\ell^\infty][\ell,~\cI_2] \simeq (\zell)^{\oplus t} \subseteq J_0(N)[\m_2]$.
By the same argument in \cite[\textsection II, Corollary 14.8]{M77} (cf. \cite[Theorem 4.2]{Yoo14}), we have $t = 1$ and $\cT[\ell^\infty][\cI_2] = \br{D_p}$. By symmetry, $\cT[\ell^\infty][\cI_3]=\br{D_q}$ and the result follows.  
\end{proof}

\section{Case where $\ell=p$ or $\ell=q$}
Throughout this section, we set $P:=p$ if $p\geq 5$; and $P:=9$ if $p=3$.
Suppose that
\begin{equation}
\ell=p \quad \text{and}\quad \begin{cases} 
\text{either } ~q \not\equiv 1 \pmod P   ~~\text{ or} \\
q \equiv 1 \pmod P ~\text{ and }~ p^{\frac{q-1}{p}} \not\equiv 1 \pmod q.
\end{cases}
\end{equation}

\begin{thm}
We have $\cT[p^\infty] = \cC[p^\infty]$.
\end{thm}
\begin{proof}
We divide the problem into three cases.
\begin{enumerate}
\item
Suppose that $q \not\equiv 1 \pmod P$ and $q \equiv 1 \pmod p$. This happens when $\ell=p=3$.
In this case, the indices of $\cI_1$, $\cI_2$ and $\cI_3$ are not divisible by $3$ (cf. \cite[Theorem 1.4]{Yoo15a}). Therefore there are no Eisenstein maximal ideals containing $3$ and
$\T_{p}/\cI_0=0$. Thus, we have $\cT[3^\infty]=\cC[3^\infty]=0$.

\item
Suppose that $q \equiv 1 \pmod P$ and $p^{\frac{q-1}{p}} \not\equiv 1 \pmod q$. Then, $\m_1=\m_2$ is not new by \cite[Theorem 3.1]{Yoo15}. Since $U_p \equiv p \equiv 0 \pmod {\m_3}$, $\m_3$ is not new. Therefore $\T_{p}/\cI_0 \simeq \T^{\old}_p/\cI_0$.
Consider the following exact sequence:
$$
\xymatrix{
0 \ar[r] & J_{\old}(\Q)[p^\infty] \ar[r] & J(\Q)[p^\infty] \ar[r] & J^{\new}(\Q)[p^\infty].
}
$$
If $J^{\new}(\Q)[p^\infty] \neq 0$, then there is a new Eisenstein maximal ideal containing $p$, which is a contradiction. Therefore we have $J_{\old}(\Q)[p^\infty] = J(\Q)[p^\infty]$. Now, the result follows from \cite[Theorem 2]{CL97}
because $p$ does not divide $2(p-1,~q-1)$.

\item 
Suppose that $q \not\equiv 1 \pmod p$. First, assume that $q \not\equiv -1 \pmod P$. 
Then, the indices of $\cI_1$, $\cI_2$ and $\cI_3$ are not divisible by $p$, there is no Eisenstein maximal ideal. Thus, $\T_p/{\cI_0}=0$ and $\cT[p^\infty]=\cC[p^\infty]=0$.

Next, assume that $q\equiv -1 \pmod P$. By the same reason as above, $\m_1$ and $\m_3$ are not maximal (but $\m_2$ is). Note that $\m_2$ is neither $p$-old nor $q$-old by Mazur. Therefore we get $\T_2/{\cI_0} \simeq \T^{\new}_{\m_2}/{\cI_0}$. 
Since $(U_p-1)(U_p+1)=(U_q-1)(U_q+1)=0$ in $\T^{\new}$, we get 
$\T_2/{\cI_0}=\T_2/{\cI_2}\simeq \T_{p}/{\cI_2}$ by \cite[Proposition 2.3]{Yoo15}. 
As in the proof of Theorem \ref{thm:maintheorem}, we get 
$$
\cT[p^\infty]=\cT[p^\infty][\cI_2]=\cC[p^\infty][\cI_2]=\cC[p^\infty].
$$
\end{enumerate}
\end{proof}

\begin{rem}
If $p>q$, then the assumption above holds and hence $\cT[p^\infty]=\cC[p^\infty]$. Since $\cC[p^\infty]=0$, there are no rational torsion points of order $p$ on $J_0(pq)$.
\end{rem}

\bibliographystyle{annotation}

\end{document}